\documentclass[a4paper,reqno]{amsart}
\usepackage{amsmath,amssymb,latexsym}
\usepackage[T1]{fontenc}
\usepackage[ansinew]{inputenc}

\theoremstyle{plain}
\newcounter{new}
\newtheorem{Th}{Theorem}
\newtheorem{Le}[new]{Lemma}
\newtheorem{Pro}[new]{Proposition}
\newtheorem{Cor}[new]{Corollary}

\theoremstyle{definition}
\newtheorem{Def}[]{Definition}

\theoremstyle{remark}
\newtheorem{Rem}[new]{Remark}

\DeclareMathOperator{\argmin}{argmin}
\DeclareMathOperator{\supp}{supp}

\DeclareMathOperator{\Lip}{Lip}

\newcommand{\N}{\mathbb{N}}
\newcommand{\R}{\mathbb{R}}
\newcommand{\Z}{\mathbb{Z}}

\newcommand{\prob}{\mathbb{P}}
\newcommand{\eps}{\epsilon}

\newcommand{\F}{\mathfrak{F}}
\newcommand{\G}{\mathcal{G}}
\newcommand{\mask}{(a_i)_{i\in\Z^s}}
\newcommand{\fa}{\mathbf{a}}

\title[Nonlinear Markov semigroups and refinement schemes]{Nonlinear Markov semigroups and\linebreak \ refinement schemes on metric spaces.}

\author{Oliver Ebner}

\address{Oliver Ebner,
         Institute of Geometry, TU Graz,
         Kopernikusgasse 24/IV, A-8010 Graz, Austria.}

\email{o.ebner@tugraz.at}

\thanks{The author is supported by the Austrian science fund, grants W1230 and P19870.}
       


\begin{document}

\begin{abstract}
This article settles the convergence question for multivariate barycentric subdivision
schemes with nonnegative masks on complete metric spaces of nonpositive
Alexandrov curvature, also known as Hadamard spaces. We establish a link
between these types of refinement algorithms and the theory of Markov chains by
characterizing barycentric subdivision schemes as nonlinear Markov semigroups.
Exploiting this connection, we subsequently prove that any such scheme
converges on arbitrary Hadamard spaces if and only if it converges for real
valued input data. Moreover, we generalize a characterization of
convergence from the linear theory, and consider approximation qualities of
barycentric subdivision schemes. A concluding section addresses the
relationship between the convergence properties of a scheme and its so-called
characteristic Markov chain.
\end{abstract}

\maketitle
\section*{Introduction and main results}

The convergence and smoothness analysis of refinement schemes processing data from manifolds and more generally metric spaces has been a subject of intense research over the last few years. As to convergence, complete spaces of nonpositive curvature have proven most accessible in terms of generalizing 
well-known facts from the linear theory to the nonlinear setting. An example of such a structure prominent in applications is the space of positive definite symmetric matrices, which represent measurements in diffusion tensor imaging.\

While the question whether the \emph{smoothness} properties of the linear model scheme prevail when passing to the nonlinear setting was successfully addressed in \cite{Grohs}, the corresponding \emph{convergence} problem remained unsolved. The aim of this article is to fill this gap in the theory.\

Relying on a martingale theory for discrete-time stochastic processes with values in negatively curved spaces developed in \cite{Sturm2}, we observe that the refinement processes in question actually act on bounded input data as nonlinear Markov semigroups. This fact will substantially facilitate their convergence analysis.\

Let us specify the general setup. Given a metric space $(X,d)$, a \emph{refinement scheme} is a map $S:\ell^\infty(\Z^s,X)\to \ell^\infty(\Z^s,X)$. We call $S$ \emph{convergent} if for all $x\in\ell^\infty(\Z^s,X)$ there exists a continuous function $S^\infty x:\R^s\to X$ such that \begin{equation}\label{eqn:converge}d_\infty(S^\infty x(\cdot/2^n),S^nx)=\sup_jd(S^\infty x(j/2^n),S^nx_j)\to 0\quad\text{ as }n\to\infty.\end{equation}
Visualizing $S^nx$ as a function on the refined grid $2^{-n}\Z^s$, convergence to $S^\infty x$ is tantamount to $d_\infty(S^nx,f|_{2^{-n}\Z^s})\to0$.  

\pagebreak

Throughout the present paper we are mostly concerned with so-called \emph{barycentric} refinement schemes associated to nonnegative real-valued $s$-variate sequences $\mask$ of finite support, henceforth referred to as \emph{masks}, which we require to fulfill the \emph{basic sum rule}

\begin{equation}\label{eqn:sumrule}\sum_{j\in\Z^s}a_{i-2j}=1\quad\text{for }i\in\Z^s.\end{equation}
Barycentric refinement schemes act on data $x\in \ell^\infty(\Z^s,X)$ from a complete metric space of nonpositive curvature in the sense of A. D. Alexandrov according to the following rule:

\begin{equation}\label{eqn:sub}Sx_i=\argmin\left(\sum_{j\in\Z^s}a_{i-2j}d^2(x_j,\,\cdot\,)\right).\end{equation}
Much is known about the convergence of these type of refinement algorithms in the case $X=\R$. On complete manifolds of nonpositive sectional curvature
convergence analysis was initiated in the article \cite{WNW}. The author in \cite{Ebner} recently proved general convergence statements for arbitrary Hadamard spaces using the principle of \emph{contractivity}: A scheme is called contractive with respect to some nonnegative function 
$D:\ell^\infty(\Z^s,X)\to\R_+$ if and only if there is $\gamma<1$ such that $$D(Sx)<\gamma D(x),\quad\text{for all}\quad x\in\ell^\infty(\Z^s,X).$$ The function $D$ is referred to as a \emph{contractivity function} for $S$. An important class of contractivity functions is associated to balanced, convex and bounded subsets $\Omega$ of $\R^s$:

\begin{equation}\label{eqn:contrfct}D_\Omega(x)=\sup_{\rho(i-j)<2} d(x_i,x_j),\end{equation}
where $\rho$ denotes the Minkowski functional of $\Omega$. Contractivity functions of this type are called \emph{admissible}, cf. \cite{Ebner}. The following result is taken from loc. cit.:

\begin{Pro}
A barycentric refinement scheme with nonnegative mask which is contractive with respect to some admissible contractivity function also converges. This implies convergence in case the support of the mask coincides with the set of lattice points within a centered unimodular zonotope or a lattice quad with nonempty interior.
\end{Pro}

The main result of the present article is a substantial extension of this statement and describes a phenomenon which could be referred to as \emph{linear equivalence}:

\begin{Th}\label{Th:main}A barycentric refinement scheme converges on arbitrary Hadamard spaces if and only if it converges on the real line.
\end{Th}

The proof of this fact, given in Section \ref{sec:convergence}, relies on a stochastic interpretation of the subdivision rule \eqref{eqn:sub}. More precisely, for each nonnegative mask $\mathbf{a}=\mask$ satisfying the basic sum rule \eqref{eqn:sumrule} one finds a so-called \emph{characteristic} Markov chain $X_n^{\mathbf{a}}$ with state space $\Z^s$ and transition matrix $(a_{i-2j})_{i,j\in\Z^s}$ in terms of which the iterates of the refinement algorithm acting on $x\in\ell^\infty(\Z^s,X)$ may be written as
$$S^nx_i=E(x\circ X^{\mathbf{a}}_n|||X^{\mathbf{a}}_0=i),$$
see Theorem \ref{Pro:main}. Here $E(\,\cdot\,|||X_0)$ denotes the \emph{filtered conditional expectation} introduced by K.-T. Sturm in \cite{Sturm2}. Thus, as in the linear case, 
$$
\begin{cases}
\N_0\to\Lip_1(\ell^\infty(\Z^s,X));\\
n\mapsto S^n
\end{cases}
$$
may be considered a (nonlinear) Markov semigroup. Here $\Lip_1(\ell^\infty(\Z^s,X))$ refers to the set of maps $T:\ell^\infty(\Z^s,X)\to\ell^\infty(\Z^s,X)$ satisfying the Lipschitz condition
\begin{equation}\label{eqn:Lip}d_\infty(Tx,Ty)\le d_\infty(x,y)\quad\text{for }x,y\in\ell^\infty(\Z^s,X).\end{equation}
Combining Theorem \ref{Th:main} with other recent developments in the theory of linear subdivision schemes with nonnegative masks and their barycentric counterparts on nonlinear objects, one comes up with a variety of remarkable results:

In the articles \cite{Zhou2} and \cite{Zhou}, X. Zhou establishes general theorems on the relation of the mask's support with its convergence properties, which, utilizing Theorem \ref{Th:main} now generalize to the following:

\begin{Th}
A barycentric subdivision scheme $S:\ell^\infty(\Z^s,X)\to \ell^\infty(\Z^s,X)$ with nonnegative mask converges under each of the following circumstances:

(i) The support of $\mask$ coincides with the set of grid points inside a balanced zonotope.

(ii) The grid dimension $s=1$ and, if, after a possible index translation, $(a_i)_{i\in\Z}=(\dots,0,0,a_0,\dots,a_N,0,0,\dots)$, the integers within the support are relatively prime and $0<a_0,a_1<1$. This also constitutes a necessary condition for convergence.
\end{Th}

Moreover, as far as finite-dimensional Hadamard manifolds are concerned, the smoothness question is settled by a combination of Theorem \ref{Th:main} and recent work from \cite{Grohs}:

\begin{Th}On a smooth Hadamard \emph{manifold}, a barycentric subdivision scheme $S:\ell^\infty(\Z^s,X)\to \ell^\infty(\Z^s,X)$ with nonnegative mask converges and produces $r$-times differentiable limit functions if and only if the same is true for the corresponding linear scheme.
\end{Th}

\section{Stochastic preliminaries}

This section is devoted to a stochastic interpretation of the subdivision rule \eqref{eqn:sub}. More precisely, we view barycentric subdivision as the semigroup
acting on $\ell^\infty(\Z^s,X)$ associated to the so-called characteristic Markov chain of $\mask$. This result requires some prerequisites about the notion of conditional expectation of random variables with values in \emph{Hadamard spaces}. 
This type of metric space was first investigated by A.D. Alexandrov in his seminal articles \cite{Alex} and \cite{Alex1}. In loc. cit., Alexandrov uses the \emph{Gauss-Bonnet} theorem, i.e. the fact that the deficiency of the angle sum in a geodesic triangle can be expressed by means of the ambient space's curvature, to generalize the notion of curvature bounds to a purely metric setting. A comprehensive introduction to the nowadays well-established theory of such spaces is \cite{Ballmann}. Furthermore we refer to the survey article \cite{Sturm} on probability measures and their centers of mass on Hadamard spaces. A treatise of the smooth case can be found in \cite{Lang}, which investigates infinite-dimensional Riemannian manifolds of nonpositive curvature.

\begin{Def}[Hadamard spaces] A complete metric space is called \emph{Hadamard-} or {global NPC}-space if and only if for $x_0,x_1$ one finds $y\in X$ such that the so-called \emph{Hadamard inequality} holds true for all $z\in X$:
\begin{equation}\label{eqn:Hadamard}d(z,y)^2\le \frac 1 2d(z,x_0)^2+\frac 1 2d(z,x_1)^2-\frac 1 4d(x_0,x_1)^2.\end{equation}
\end{Def}

In other words, a Hadamard space is a complete metric space with nonpositive curvature in the sense of A. D. Alexandrov. The Hadamard inequality \eqref{eqn:Hadamard} expresses the fact that geodesic triangles are 'slim' compared to Euclidean ones with the same edge lengths.

\pagebreak

Recall that a continuous curve $x_t\in X$, $0\le t\le1$, is called \emph{geodesic} if and only if $d(x_s,x_t)=d(x_0,x_1)|t-s|$ for all $0\le s,t\le 1$. A well-known property of Hadamard spaces is that they are \emph{strongly geodesic}: given any two points $x_0,x_1\in X$, one finds a unique geodesic joining them.
In particular, there is meaningful notion of convexity. More precisely, a subset $K\subseteq X$ of a global NPC space is called \emph{convex} if and only if for each two $x_0,x_1\in K$, the joining geodesic $x_t$ remains in $K$. Nonpositive curvature turns out to have a major impact on the convexity properties of the metric $d(\,\cdot\,,\,\cdot\,)$. 
As announced above, we are particularly interested in the concept of conditional expectation for random variables with values in Hadamard spaces.  Although several approaches appear in the literature, a definition due to K.-T. Sturm, see \cite{Sturm2}, serves our purposes best. 
Sturm's definition relies on a convex projection property enjoyed by NPC spaces:

\begin{Pro}\label{Pro:convexprojection} Suppose $(X,d)$ is a global NPC space, and $C\subseteq X$ a closed and convex subset. Then there exists a well-defined \emph{projection map} $\pi_C:X\to C$ defined by the relation
$$d(\pi_C(x),x)=\min_{y\in C}d(y,x).$$
This projection is Lipschitz-continuous in a sense that $d(\pi_C(x),\pi_C(y))\le d(x,y)$ for $x,y\in X$.
\end{Pro}

Given a probability space $(\Omega,\F,\prob)$ and a metric space $(X,d)$, a strongly measurable function is called \emph{square-integrable} if and only if $$\int_\Omega d^2(f(\omega),x)\prob(d\omega)<\infty$$ for one (and then all) $x\in X$. The space $L^2(\F,X)$ of square-integrable functions $f:\Omega\to X$ factorized by the relation of being equal almost surely is endowed with a metric
$$d_2(f,g)=\left(\int_\Omega d^2(f(\omega),g(\omega))\prob(d\omega)\right)^{\frac 1 2}.$$
In case $(X,d)$ is a global NPC space, it is well-known that $(L^2(\F,X),d_2)$ inherits the Hadamard property. Moreover, given a subalgebra $\G\subseteq\F$, it is easy to show that $L^2(\G,X)$, the subspace of $L^2(\F,X)$ corresponding to $\G$-measurable square-integrable maps, is closed and convex. In view of Proposition \ref{Pro:convexprojection} one thus obtains the following Definition:

\begin{Def}[\cite{Sturm2}]\label{Def:cond}
Suppose $(\Omega,\F,\prob)$ is a probability space, and let $Y\in L^2(\F,X)$ be a square-integrable random variable with separable image in the Hadamard space $X$. Given a subalgebra $\G\subseteq\F$, there is a $\G$-measurable $Z:\Omega\to X$ minimizing the functional
$$Z'\mapsto E(d^2(Y,Z'))$$
among all $\G$-measurable square-integrable random variables. Any other minimizer coincides with $Z$ almost surely.
One refers to $Z$ as the \emph{conditional expectation of $Y$ given $\G$} and uses the notation $Z=E(Y|\G)$.
\end{Def}

Otherwise put, following the notation of Proposition \ref{Pro:convexprojection}, $E(Y|\G)=\pi_{L^2(\G,X)}(Y)$. This definition follows the principle that in the real-valued case, the conditional expectation $E(\,\cdot\,|\G):L^1(\F,\R)\to L^1(\G,\R)$ as introduced by Kolmogorov
restricts to $L^2$ as the orthogonal projection to the space of $\G$-measurable $L^2$-functions.\

\pagebreak

\begin{Rem}
The space $L^p(\F,X)$ is, as in the case $p=2$, defined to be the space of strongly measurable functions $f:(\Omega,\F,\prob)\to X$ such that 
$\int_\Omega d^p(f(\omega),x)\prob(d\omega)$ modulo equality almost surely. Again $L^p(\F,X)$ comes with a metric of the form
$$d_p(f,g)=\left(\int_\Omega d^p(f(\omega),g(\omega))\prob(d\omega)\right)^{\frac 1 p}.$$
The conditional expectation as defined above is continuous in a sense that for each two $Y,Z\in L^2(\F,X)$ and $p\in[1,\infty]$,
$$d_p(E(Y|\G),E(Z|\G))\le d_p(Y,Z).$$
In particular, $E(\,\cdot\,|\G)$ extends continuously to $L^1(\F,X)$.
\end{Rem}

\begin{Rem}
It is inherent in the existence statement provided by Definition \ref{Def:cond} that each integrable random variable with values in $X$ possesses an \emph{expected value} defined by
$$E(Y):=E(Y\mid\{\emptyset,\Omega\}),$$
which is the unique minimizer of the functional
$$z\mapsto\int_\Omega d^2(Y(\omega),z)\prob(d\omega).$$
\end{Rem}

Recall the elementary \emph{smoothing} property of conditional expectations for random variables with values in a linear space: Given an ordered pair of subalgebras $\G\subseteq\widetilde\G\subseteq\F$, one has
$$E(E(\,\cdot\,|\widetilde\G)|\G)=E(\,\cdot\,|\G)$$
almost surely. Not surprisingly, this associativity property is violated in the nonlinear case, see Example 3.2 in \cite{Sturm2}. Actually a generalization of this feature would render the theory of nonlinear subdivision schemes of type \eqref{eqn:sub} obsolete, as the discussion following Theorem \ref{Th:lin} below illustrates.

It is for the reason of lacking associativity that K.-T. Sturm in his treatise \cite{Sturm2} defines:

\begin{Def}
Suppose $\F_0\subseteq\F_1\subseteq\cdots\F_N=\F$ is a sequence of subalgebras. Furthermore assume $Y\in L^2(\F,X)$. Then one defines the \emph{filtered conditional expectation} of $Y$ given $(\F_n)_{0\le n\le N}$ as
\begin{equation}\label{eqn:smoothing}E(Y|||\F_0)=E(\cdots E(E(Y|\F_{N-1})|\F_{N-2})\cdots|\F_0).\end{equation}
\end{Def}

Let us briefly recall the basic construction of a Markov chain, beginning with a foundational definition. Recall that  a family of maps
$$p_{n,m}:\Z^s\times\Z^s\to\R$$
parametrized by nonnegative integers $n\le m$, is called \emph{Markov transition kernel} if the following are fulfilled:
\begin{enumerate}
\item[(i)] $p_{n,m}(i,\,\cdot\,)$ is a probability measure for all $i\in\Z^s$
\item[(ii)] $p_{n,n}(i,\,\cdot\,)=\delta_{\{i\}}$.
\item[(iii)] The \emph{Chapman-Kolmogorov} equations are fulfilled: for each $n\le\ell\le m$ one has
\begin{equation}\label{eqn:Chapman}p_{n,m}(i,j)=\sum_{k\in\Z^s}p_{n,\ell}(i,k)p_{\ell,m}(k,j).\end{equation}
\end{enumerate}
We use the notation $P=(p_{n,m})_{n\le m}$.

Let $\Omega=(\Z^s)^{\N_0}$ denote the space of sequences on the grid $\Z^s$, endowed with the infinite power $\F=\bigotimes_{n\in\N}\mathfrak{P}(\Z^s)$ of the discrete sigma algebra on $\Z^s$. Consider the filtration
$$\F_n=\sigma(\{i_0\}\times\dots\times\{i_n\}\times\Z^s\times\cdots\mid i_j\in\Z^s\text{ for }j=1,\dots,n).$$
Choose an initial distribution $\alpha$ on $\Z^s$, and introduce a probability measure on $\F$ via 
$$\prob_\alpha(\{i_0\}\times\dots\times\{i_n\}\times\Z^s\times\cdots)=\alpha_{i_0}p_{0,1}(i_0,i_1)p_{1,2}(i_1,i_2)\dots p_{n-1,n}(i_{n-1},i_n),$$
and the standard extension theorems. Whenever $\alpha$ equals $\delta_{\{i\}}$, the point measure on $\{i\}$, we will write $\prob_\alpha=\prob_i$.
Expected values of integrable random variables $Y:\Omega\to X$ with respect to $\prob_\alpha$ are, as usual, written as 
$$E_\alpha(Y)=\argmin\Big(\int_\Omega d^2(Y(\omega),\,\cdot\,)\prob_\alpha(d\omega)\Big).$$
Finally, the discrete stochastic process 
$$X_n:(\Omega,\F,\prob)\to\Z^s;\quad\omega\mapsto\omega_n$$
constitutes a \emph{Markov chain} associated to the transition kernel $P$, meaning that the \emph{linear Markov property} holds true: For any nonnegative $f:\Z^s\to\R$ one has
\begin{equation}\label{eqn:linmarkov}E_\alpha(f(X_m)|\F_n)=\sum_{j\in\Z^s}p_{n,m}(X_n,j)f(j).\end{equation}
In particular, this linear Markov property allows for the interpretation of $p_{n,m}(i,j)$ as the transition probability $\prob_\alpha(X_m=j|X_n=i)$.

In view of the convergence analysis of barycentric subdivision schemes, it is of particular interest to gain a deeper understanding of principle of conditioning in case the filtration stems from a Markov chain.
A \emph{nonlinear Markov property} analogous to \eqref{eqn:linmarkov}, see \cite[Theorem 5.2]{Sturm2}, leads to a representation of the conditional expectation explicit enough for our purposes. We provide a short proof adapted to our setting, beginning with an auxiliary result which can be found e.g.\ in \cite{Sturm2}:

\begin{Le}\label{Le:prepcond}
Suppose $(X_k)_{k\in\N_0}$ is a Markov chain in $\Z^s$ associated to the transition kernel $P$. Choose an initial distribution $\alpha$. Furthermore assume $Y:\Omega\to X$ is $\F_n$-measurable, and let $x\in\ell^\infty(\Z^s,X)$.
Then for nonnegative and measurable $f:X\times X\to\R$ and $m\geq n$ we have
$$\int_\Omega f(x\circ X_m(\omega),Y(\omega))\prob_\alpha(d\omega)=\int_\Omega \sum_j p_{n,m}(X_n(\omega),j)f(x_j,Y(\omega))\prob_\alpha(d\omega).$$
\end{Le}

\begin{Pro}[Nonlinear Markov property]\label{Le:cond}Let $(X_k)_{k\in\N_0}$ be a Markov chain as in Lemma \ref{Le:prepcond}, and suppose $x\in\ell^\infty(\Z^s,X)$, with $(X,d)$ Hadamard. Choose $n,m\in\N_0$ with $n<m$. Then
\begin{equation}\label{eqn:nonlinmarkov}
\begin{split}E_\alpha(x(X_m)|\F_n)(\omega)&=\argmin\sum_{j\in\Z^s}p_{n,m}(X_n(\omega),j)d^2(x_j,\,\cdot\,)\\
&=E_{X_n(\omega)}(x(X_m)).
\end{split}
\end{equation}
\end{Pro}
\begin{proof}
By the linear Markov property \eqref{eqn:linmarkov}, 
\begin{align*}
Y(\omega)&:=\argmin\left(E_\alpha(d^2(x\circ X_m,\,\cdot\,)\mid\F_n)(\omega)\right)\\
&=\argmin\sum_{j\in\Z^s}p_{n,m}(X_n(\omega),j)d^2(x_j,\,\cdot\,).
\end{align*}
Clearly $Y$, as a measurable function of $X_n$, is $\F_n$-measurable. Thus, in order to verify that $Y$ is indeed the conditional expectation of $x(X_m)$ given $\F_n$, it remains to show that for each $\F_n$-measurable function $Z$ the inequality $E_\alpha(d^2(X_m,Y))\le E_\alpha(d^2(X_m,Z))$ holds true, cf.\ Definition \ref{Def:cond}. For this sake, define
$$\begin{cases}\psi:\Z^s\times X\to\R_{\geq0}\cup\{\infty\};\\
(i,z)\mapsto\sum_j p_{n,m}(i,j)d(x_j,z).\end{cases}$$
By construction of $Y$ we have $\psi(X_m,Y)\le\psi(X_m,Z)$.
Thus, Lemma \ref{Le:prepcond} implies 
\begin{align*}E_\alpha(d^2(x\circ X_m,Y))&=E_\alpha(\psi(X_m,Y))\\
&\le E_\alpha(\psi(X_m,Z))\\
&=E_\alpha(d^2(x\circ X_m,Z)).
\end{align*}
\end{proof}
\begin{Rem}Proposition \ref{Le:cond} implies that the expression $E_\alpha(x(X_m)|\F_n)$ actually is independent of the initial distribution $\alpha$.
Therefore we will omit $\alpha$ in the following and simply write $E(x(X_m)|\F_n)$.
\end{Rem}

We are now in a position to establish a link between nonlinear Markov semigroups and barycentric refinement processes.
Suppose $\fa=\mask$ is a nonnegative compactly supported $s$-variate sequence such that $\sum_ja_{i-2j}=1$ for all $i\in\Z^s$. Define recursively
$a^{0}_i=\delta_{\{0\}}(i)$ and
$$a^{(n+1)}_i=\sum_{j\in\Z^s}a_{i-2j}a^{(n)}_j.$$
Then clearly 
$$p_{n,m}^\fa(i,j)=a^{(m-n)}_{i-2^{m-n}j}$$
defines a Markov transition kernel. This kernel is \emph{homogeneous} in the sense that $p_{n,n+1}^\fa(i,j)=a_{i-2j}$ for any $n\in\N_0$. We write $P^\fa=(p_{n,m}^\fa)_{n\le m}$, denote the associated Markov chain by $X^\fa_n$, and refer to $X^\fa_n$ as the \emph{characteristic Markov chain} for $\mask$. 
The central observation of this article is the following consequence of the nonlinear Markov property \eqref{eqn:nonlinmarkov}:

\begin{Th}\label{Pro:main}
Suppose $x:\Z^s\to X$ is bounded, where $(X,d)$ is a Hadamard space. Let $S$ be a barycentric refinement scheme acting on data from $X$ according to the subdivision rule \eqref{eqn:sub}. Let $X_n^\fa$ denote the characteristic Markov chain of $\mask$. Then
$$S^nx\circ X_0^\fa=E(x\circ X_n^\fa|||\F_0).$$
\end{Th}
\begin{proof}
This statement is proven by induction over $n$ using the following computation and Proposition \ref{Le:cond}:

\begin{align*}
E(S^{n-k}x\circ X^\fa_k\mid \F_{k-1})&=\argmin\Big(\sum_{j\in\Z^s}p_{k-1,k}(X^\fa_{k-1},j)d^2(S^{n-k}x_j,\cdot)\Big)\\
&=\argmin\Big(\sum_{j\in\Z^s}a_{X^\fa_{k-1}-2j}d^2(S^{n-k}x_j,\cdot)\Big)\\
&=S^{n-k+1}x\circ X^\fa_{k-1}.
\end{align*}

\end{proof}

The remainder of this article is devoted to analyzing the effects of this representation of the iterates of $S$ on the convergence properties of barycentric schemes with nonnegative masks.

\section{The convergence problem}\label{sec:convergence}

We begin this section by summarizing some well-known facts about the convergence of barycentric schemes acting on real-valued input data. Classical resources on this topic are \cite{Prautzsch,CDM,Reif}. 

\begin{Th}\label{Th:lin}
Suppose $\fa=\mask$ is an $s$-variate compactly supported sequence of nonnegative reals. Define
a refinement scheme $\widetilde S:\ell^\infty(\Z^s,\R)\to\ell^\infty(\Z^s,\R)$ via
$$\widetilde Sx_i=\sum_{j\in\Z^s}a_{i-2j}x_j,\quad\text{where }x\in\ell^\infty(\Z^s).$$
Then a necessary condition for the convergence of $\widetilde S$ on $\R$ is the basic sum rule \eqref{eqn:sumrule}.
In case the mask $\mask$ obeys this rule, we conclude
\begin{align*}\widetilde Sx_i&=\argmin\left(\sum_{j\in\Z^s}a_{i-2j}|x_j-\,\cdot\,|^2\right)=\argmin\left(\sum_{j\in\Z^s}a_{i-2j}d_{|\cdot|}(x_j,\,\cdot\,)^2\right).\end{align*}

Moreover, $\widetilde S$ converges if and only if there exists a continuous $\varphi:\R\to\R$ subject to the functional equations

\begin{eqnarray}
\label{eqn:refine}\varphi(t)&=&\sum_ja_j\varphi(2t-j)\\
\label{eqn:pou}\sum_j\varphi(t-j)&=&1.
\end{eqnarray}

Due to Equation \eqref{eqn:refine}, $\varphi$ is referred to as an $\fa$-refinable function.
Given bounded, real-valued input data $(x_i)_{i\in\Z^s}$, the limit function may be written as
$$\widetilde S^\infty x(t)=\sum_{j\in\Z^s}\varphi(t-j)x_j.$$
In particular, $\varphi=\widetilde S^\infty\delta_{\{0\}}$, where $\delta_{\{0\}}$ denotes the Dirac distribution on the origin.
\end{Th}

Assuming that conditional expectations of bounded random variables mapping to the metric space $X$ are well-defined in the sense of Definition \ref{Def:cond}, and in addition satisfy the smoothing property \eqref{eqn:smoothing}, we could deduce from Theorem \ref{Pro:main}
\begin{align*}S^nx\circ X_0&=E(x\circ X^\fa_n|||\F_0)\\
&=E(x\circ X^\fa_n|\F_0)\\
&=\argmin(\sum_jp^{(n)}(X_0,j)d^2(x_j,\,\cdot\,)).\end{align*}

Note, however, that $p^{\mathbf{a}}_n(i,j)$, the $n$-step transition probabilities of $(X^\fa_k)_{k\in\N_0}$, can be viewed as $(\widetilde S^n\delta_0)_{i-2^nj}$, where $\widetilde S$ denotes the linear counterpart to $S$, and $\delta_0$ the Dirac delta on the origin, cf.\ Theorem \ref{Th:lin}. Thus, the assumption of associativity would guarantee that every scheme converging for linear input data would converge on $X$ as well. Indeed, the limit functions for given input data $x\in\ell^\infty(\Z^s,X)$ would satisfy
$$S^\infty x (t)=\argmin(\sum_j\varphi(t-j)d^2(x_j,\,\cdot\,)),$$
where $\varphi=\widetilde S^\infty\delta_{\{0\}}$, leading to a complete analogy to the linear case.
However, nonlinear conditioning is a non-associative operation. The above observations demonstrate that this lack of property \eqref{eqn:smoothing} constitutes the need for a further discussion of convergence.\

The first result of this section is a small, but useful generalization of Theorem 1 in \cite{Ebner}. Although the proof transcribes more or less directly, we give a detailed exposition for the reader's convenience.

\begin{Th}\label{Th:converge}
Let $S,T$ be refinement schemes acting on data from a Hadamard space $X$. Then $S$ converges under the following assumptions:

(i) There is a function $D:\ell^\infty(\Z^s,X)\to\R_{\geq0}$, a nonnegative real number $\gamma<1$ and a positive integer $n_0$ such that 
\begin{equation}\label{eqn:contr}D(S^nx)\le \gamma^{[n/n_0]}D(x)\end{equation}
for $x\in \ell^\infty(\Z^s,X)$ and $n\in\N$.

(ii) $T$ is convergent and satisfies
$$d_\infty(Tx,Ty)\le d_\infty(x,y)$$
for $x,y\in \ell^\infty(\Z^s,X)$.

(iii) There is $C\geq0$ such that
$$d_\infty(Sx,Tx)\le C\cdot D(x)$$
for $x\in \ell^\infty(\Z^s,X)$.
\end{Th}
\begin{proof}

We set $f_n(y):=T^\infty(S^nx)(2^ny)$ and claim that this defines a Cauchy sequence in $\left(C(\R^s,X),d_\infty\right)$. Note first that given $n\in\N$ and  $y\in\R^s$, by continuity of $f_n$ respectively $f_{n+1}$, we find $j\in\Z^s$ and $m\in\N$ such that
\begin{equation}\label{eqn:1}d(f_r(y),f_r(2^{-m}j))<C\cdot D(x)\gamma^{[n/n_0]}\quad\text{ for }r=n,n+1.\end{equation}
Moreover, due to convergence of $T$, by multiplying both the numerator and the denominator of the number $j/2^m$ with a power of two if necessary we may assume $m$ to be sufficiently large for 
\begin{align*}d(f_r(2^{-m}j),T^{m-r}(S^rx)_j)&=d(T^\infty S^rx(2^{r-m}j),T^{m-r}(S^rx)_j)<C\cdot D(x)\gamma^{[n/n_0]}\end{align*}
to hold for $r=n,n+1$, in addition to \eqref{eqn:1}. This together with (i) and (ii) implies
\begin{align*}
d(f_n(y),f_{n+1}(y))&\le d( f_n (y), f_n (2^{-m}j))\\
&+d( f_n (2^{-m}j),T^{m-n}S^nx_j)\\
&+d(T^{m-n}S^nx_j,T^{m-n-1}S^{n+1}x_j)
\\&+d(T^{m-n-1}S^{n+1}x_j, f_{n+1} (2^{-m}j))\\
&+d( f_{n+1} (2^{-m}j), f_{n+1} (y))\\
&<5C\cdot D(x)\gamma^{[n/n_0]},
\end{align*}
showing that $f_n$ is a Cauchy sequence. Since $X$ is complete, we find a continuous $f:\R^s\to X$ with $f_n\to f$ uniformly. We claim that $S^nx$ converges to $f$ in the sense of \eqref{eqn:converge}. For $m\geq n$ and $j\in\Z^s$, we obtain the inequality

\begin{align*}
d(T^{m-n}S^nx_j,S^mx_j)&\le\sum_{k=n}^{m-1}d(T^{m-k}S^kx_j,T^{m-k-1}S^{k+1}x_j)\\&
\le\sum_{k=n}^{m-1}\gamma^{[k/n_0]}\cdot D(x)C\le\gamma^{[n/n_0]}\left(\frac{n_0D(x)C}{1-\gamma}\right),\end{align*}
which together with $$d(f_n(2^{-m}j),S^mx_j)\le d(f_n(2^{-m}j),T^{m-n}S^nx_j)+d(T^{m-n}S^nx_j,S^mx_j)$$
establishes the claim.
\end{proof}

\begin{Def}
In accordance with \cite{Ebner}, we call a scheme $S$ satisfying \eqref{eqn:contr} \emph{weakly contractive}. Thus, a weakly contractive scheme is contractive if and only if $n_0=1$. 
\end{Def}

\pagebreak

In the following we rely on a nonlinear version of Jensen's inequality due to K. T. Sturm, cf.\ \cite{Sturm2}:

\begin{Le}[Conditional Jensen's inequality]\label{Le:Jensen}
Suppose $\psi:X\to\R$ is a convex, lower semicontinuous function on a Hadamard space $(X,d)$, and $(\Omega,\F,\prob)$ is a probability space. Moreover suppose $(\F_k)_{k\in\N_0}$ is a filtration in $\F$. Then for each bounded, $\F_N$-measurable random variable $Y:\Omega\to X$ the following holds true:
\begin{equation}\label{eqn:Jensen}\psi(E(Y|||(\F_k)_{k\geq n}))\le E(\psi(Y)|\F_n).\end{equation}
\end{Le}

\begin{Le}\label{Le:contr}
Suppose the linear scheme associated to $\mask$ converges, and $\supp(\fa)$ $\subseteq\Omega$, with $\Omega$ bounded, convex and balanced. Denote by $\rho:\R^s\to\R_{\geq0}$ the Minkowski functional of $\Omega$. Furthermore define $D:\ell^\infty(\Z^s,X)\to\R_{\geq0}$ via 
$$D(x)=\sup_{\rho(i-j)<2}d(x_i,x_j).$$
Then the barycentric scheme $S$ associated to $\mask$ is weakly contractive with respect to $D$.
\end{Le}
\begin{proof}
The Hadamard property implies that for each $z_0\in X$ the function $$X\to\R_{\geq0};\quad z\mapsto d(z,z_0),$$
which clearly is continuous, is convex as well.  Thus, by Jensen's inequality \eqref{eqn:Jensen} and Theorem \ref{Pro:main},

\begin{equation}\label{eqn:pre}d(S^nx\circ X_0,z_0)=d(E(x\circ X^\fa_n|||\F_0),z_0)\le E(d(x\circ X^\fa_n,z_0)|\F_0).\end{equation}
Recall that the transition kernel of $X^\fa_n$ takes the form 
$$P^\fa=\left(a^{(m-n)}_{i-2^{m-n}j}\right)_{n\le m}.$$ Thus Proposition \ref{Le:cond} implies $E(d(X^\fa_n,z_0)|\F_0)=\sum_{k\in\Z^s}a_{X^\fa_0-2^nk}^{(n)}d(x_k,z_0)$. Together with \eqref{eqn:pre} this gives

$$d(S^nx_i,z_0)\le\sum_{k\in\Z^s}a^{(n)}_{i-2^nk}d(x_k,z_0)\quad\text{for all }i\in\Z^s.$$
Substituting $S^nx_j$ for $z_0$, we deduce

\begin{align*}d(S^nx_i,S^nx_j)&\le\sum_{k\in\Z^s}a^{(n)}_{i-2^nk}d(x_k,S^nx_j)\\
&\le\sum_{k\in\Z^s,\ell\in\Z^s}a^{(n)}_{i-2^nk}a^{(n)}_{j-2^n\ell}d(x_k,x_\ell).\end{align*}

The fact that the support of $\mask$ is contained in the balanced, convex and bounded set $\Omega$ together with the recursion $a^{(n)}_i=\sum_ja_{i-2j}a^{(n-1)}_j$ (which amounts to the Chapman-Kolmogorov equations \eqref{eqn:Chapman}) implies $\supp(a^{(n)})\subseteq(2^n-1)\Omega$, see also \cite{CDM}. 

Since the linear subdivision scheme with mask $\mask$ converges, we find a refinable function $\varphi:\R^s\to\R$ satisfying \eqref{eqn:refine} and \eqref{eqn:pou}, cf. Theorem \ref{Th:lin}.
Recall that one obtains this refinable function as the limit of the linear scheme acting on the input data $y_j=\delta_{j0}$, cf.\ Theorem \ref{Th:lin} : \begin{equation}\label{eqn:linconverge}\sup_i|a^{(n)}_i-\varphi(i/2^n)|=\eps_n\underset{n\to\infty}{\longrightarrow}0.\end{equation}

\pagebreak
\noindent Accordingly, setting $A=\{(k,\ell)\in\Z^s\times\Z^s\mid\max(\rho(i-2^nk),\rho(j-2^n\ell))\le2^n-1\}$, we obtain

\begin{equation}\label{eqn:aux1}\begin{split}
d(S^nx_i,S^nx_j)&\le\sum_{(k,\ell)\in A}a^{(n)}_{i-2^nk}a^{(n)}_{j-2^n\ell}d(x_k,x_\ell)\\
&\le\sum_{(k,\ell)\in A}\varphi(i/2^n-k)\varphi(j/2^n-\ell)d(x_k,x_\ell)\\
&\quad+\eps_n\left(\sum_{(k,\ell)\in A}(a^{(n)}_{i-2^nk}+a^{(n)}_{j-2^n\ell})d(x_k,x_\ell)\right)\\
&\quad+\eps_n^2\left(\sum_{(k,\ell)\in A}d(x_k,x_\ell)\right)
\end{split}\end{equation}
Now, if $i,j,k,\ell\in\Z^s$ are such that $\rho(i-j)<2$, $\rho(i-2^nk)\le2^n-1$ and $\rho(j-2^n\ell)\le2^n-1$, one concludes
\begin{equation}\label{eqn:aux2}\begin{split}\rho(k-\ell)&\le\frac{1}{2^n}(\rho(i-2^nk)+\rho(i-j)+\rho(j-2^n\ell)\\&
<\frac{1}{2^n}(2(2^n-1)+2)=2.\end{split}\end{equation}
Define
$$\psi(s,t)=\sum_{i\in\Z^s}\varphi(t-i)\varphi(s-i).$$
Then, since the refinable function is uniformly continuous, the property \eqref{eqn:pou} implies that for $n$ large enough,
\begin{equation}\label{eqn:aux3}\alpha_n=\inf_{\rho(t-s)<2^{-n+1}}\psi(s,t)>\eps>0.\end{equation}
By boundedness of $\Omega$ we also obtain
\begin{equation}\label{eqn:aux4}M=\sup_{t\in\R^s}|\Z^s\cap(t+\Omega)|<\infty.\end{equation}
Combining \eqref{eqn:pou} with \eqref{eqn:aux1} through \eqref{eqn:aux4} further gives

\begin{equation}\label{eqn:unif}D(S^nx)=\sup_{\rho(i-j)<2}d(S^nx_i,S^nx_j)\le\gamma_nD(x),\end{equation}
where $\gamma_n=(1-\alpha_n+2\eps_n+M^2\eps_n^2)$. Clearly, for $n_0$ large enough, $\gamma=\gamma_{n_0}<1$. Moreover, the estimate \eqref{eqn:unif} is uniform in $x$ (and even $d$). The same argument leading to the first inequality in \eqref{eqn:aux1} together with \eqref{eqn:aux2} provides
$$D(S^mx)\le D(S^kx)\quad\text{for }m\geq k.$$
Thus, for $n\in\N$ one concludes:

\begin{align*}
D(S^nx)&\le\gamma D(S^{n-n_0}x)\\
&\le\gamma^{\left[n/n_0\right]} D(S^{n-n_0\left[n/n_0\right]})\\
&\le\gamma^{\left[n/n_0\right]} D(x),
\end{align*}
which completes the proof.
\end{proof}

Recall that the \emph{tensor product} $(a\otimes b)_{i\in\Z^{s+t}}$of two masks $\mask$ and $(b_j)_{j\in\Z^t}$ is defined by 
$$(a\otimes b)_{(i,j)}=a_i\cdot b_j.$$

\pagebreak

\begin{Le}[\cite{Ebner}]\label{Le:B-spline}Suppose $S$ and the corresponding $D$ are as in Lemma \ref{Le:contr}. Define $(b_i)_{i\in\Z}$ via $b_0=1$, $b_{-1}=b_1=\textstyle\frac 1 2$, and $b_i=0$ for $|i|>1$. Let $T$ denote the barycentric scheme associated to the $s$-fold tensor power $b\otimes\dots\otimes b$. Then $T$ is Lipschitz in the sense of \eqref{eqn:Lip} and converges on any Hadamard space. Moreover, there is $C>0$ such that $d_\infty(Sx,Tx)\le C\cdot D(x)$ for all $x\in\ell^\infty(\Z^s,X)$.
\end{Le}

\begin{proof}[Proof of Theorem \ref{Th:main}]
Suppose $S$ denotes the barycentric scheme associated to the nonnegative mask $\mask$. Under the assumption that the linear counterpart of $S$ converges, combining Lemmas \ref{Le:contr} and \ref{Le:B-spline}, we find a function $D:\ell^\infty(\Z^s,X)\to\R_{\geq0}$, a convergent scheme $T:\ell^\infty(\Z^s,X)\to\ell^\infty(\Z^s,X)$, and constants $\gamma<1$ and $C\geq0$ such that
\begin{itemize}
\item[(i)] There is a positive integer $n_0$ such that $D(S^nx)\le \gamma^{[n/n_0]}D(x)$
for $x\in \ell^\infty(\Z^s,X)$ and $n\in\N$.

\item[(ii)]$T\in\Lip_1(\ell^\infty(\Z^s,X))$ is convergent.

\item[(iii)] $d_\infty(Sx,Tx)\le C\cdot D(x)$ for $x\in \ell^\infty(\Z^s,X)$.
\end{itemize}
Thus, by Theorem \ref{Th:converge}, the scheme $S$ converges.
\end{proof}

A well-known result from the linear theory is the following:

\begin{Pro}\label{Pro:contrconverge}
The \emph{univariate} and \emph{linear} scheme $\widetilde S$ associated to the mask $(a_i)_{i\in\Z}$ converges if and only if there is $\gamma<1$ and $C\geq0$ such that 
$$\sup_{i\in\Z}|\widetilde S^nx_i-\widetilde S^nx_{i+1}|\le C\cdot\gamma^n \sup_{i\in\Z}|x_i-x_{i+1}|\quad\text{ for all }n\in\N_0.$$
\end{Pro}

Theorems \ref{Th:main} and \ref{Th:converge} along with Lemma \ref{Le:contr} put us in a position to generalize this statement to the setting of Hadamard spaces. Still need an easy auxiliary result.

\begin{Le}\label{Le:canonicalcontractivity}Suppose $(X,d)$ is a metric space, and let $$D_\infty(x)=\sup_{\|i-j\|_\infty\le 1}d(x_i,x_j).$$ Then a refinement scheme $S$ is weakly contractive with respect to an admissible contractivity function if and only if there is $\gamma<1$ and $C\geq0$ such that
$$D_\infty(S^nx)\le C\gamma^nD_\infty(x)$$
\end{Le}
\begin{proof}
Suppose $S$ is weakly contractive with respect to $D_\Omega$, meaning there is $n_0\in\N$ and $\widetilde\gamma<1$ such that $D\circ S^n\le\widetilde\gamma^{(n/n_0)}D$. It is not difficult to see (cf. \cite{Ebner}) that there are $r,R>0$ such that
$$rD_\Omega\le D_\infty\le RD_\Omega.$$
Observe that, since $[n/n_0]\geq n/n_0-1$ one has $\widetilde\gamma^{[n/n_0]}\le\widetilde\gamma^{n/n_0-1}=\widetilde C\gamma^n$, where
$\widetilde C=\widetilde\gamma^{-1}$ and $\gamma=\widetilde\gamma^{1/n_0}<1$. Moreover define $C=\frac{R\widetilde C} r$. Then
\begin{align*}
D_\infty\circ S^n&\le RD_\Omega\circ S^n\le R\widetilde\gamma^{[n/n_0]}D_\Omega\\
&\le\frac R r \widetilde C \gamma^nD_\infty=C\gamma^nD_\infty.
\end{align*}
Now assume there is $\gamma<1$ and $C\geq0$ such that
$$D_\infty\circ S^n\le C\gamma^nD_\infty.$$
Choose $n_0\in\N$ such that $\frac {C R} r \gamma^{n_0}\le1$. From \eqref{eqn:Jensen} it follows that $D_\Omega\circ S^n\le D_\Omega$. On the other hand, for $n\geq n_0$ we have $[n/n_0]\le n-n_0$ and thus 
\begin{align*}
D_\Omega\circ S^n&\le \frac 1 r D_\infty\circ S^n\le \frac C r \gamma^nD_\infty\\
&\le\left(\frac {RC} r\gamma^{n_0}\right) \gamma^{n-n_0}D_\Omega\le\gamma^{[n/n_0]}D_\Omega.
\end{align*}
\end{proof}

We are now able to generalize Proposition \ref{Pro:contrconverge}:
\begin{Th}The refinement scheme associated to $\mask$ converges on arbitrary Hadamard spaces if and only if there is $C\geq0$ and $\gamma<1$ such that for all $(X,d)$ Hadamard
$$D_\infty^X(S^nx)\le C\cdot\gamma^n D_\infty(x)\quad\text{ for all }x\in\ell^\infty(\Z^s,X),$$
where, as above, $D_\infty(x)=\sup_{\|i-j\|_\infty\le1}d(x_i,x_j)$.
\end{Th}
\begin{proof}
This follows from combining Lemma \ref{Le:canonicalcontractivity} with Lemma \ref{Le:contr} and Theorem \ref{Th:converge}.
\end{proof}

We conclude this section with an approximation result for Lipschitz functions:

\begin{Th}
Suppose $f:(\R^s,\|\cdot\|)\to (X,d)$ is Lipschitz-continuous with constant $C>0$, and $S$ is a convergent barycentric scheme whose mask is supported on $\{x\in\R^s\mid\|x\|\le r\}$. Sample $f$ on the grid $h\Z^s$, $h>0$, via $x_i=f(hi)$. Then 
$$d_\infty(S^\infty x(h^{-1}\cdot),f(\cdot))\le rC\cdot h.$$
\end{Th}
\begin{proof}
Suppose $n\in\N_0$ and $i\in\Z^s$. Then by Lemma \ref{Le:Jensen} and Theorem \ref{Pro:main} we obtain
\begin{align*}d(S^nx_{2^{n-k}i},f(hi/2^k))&\le\sum_{\|2^{n-k}i-2^nj\|\le (2^n-1)r}a^{(n)}_{2^{n-k}i-2^nj}d(x_j,f(hi/2^k))\\
&\le\sup_{\|i/2^k-j\|\le (1-2^{-n})r}d(f(hj),f(hi/2^k))\\
&\le rC\cdot h,\end{align*}
from which the claim follows.
\end{proof}

\section{$L^p$-convergence of the characteristic Markov chain}

This short section clarifies the relationship between the stochastic convergence of the Markov chain associated to $\mask$ and its counterpart in the theory of barycentric subdivision schemes. 

\begin{Le}\label{Le:Ball}
Suppose $\supp(\fa)\subseteq C\cap\Z^s$, where $C$ is a convex, balanced, and compact set. Let $\rho:\R^s\to\R_{\geq0}$ denote the Minkowski functional of $C$. Recall the notation $\prob_i$ for the probability measure on $(\Z^s)^{\N_0}$ induced by the transition kernel $P^\fa$ and the initial distribution $\delta_{\{i\}}$. Then $$\rho(i)\le2^n\quad\Longrightarrow\quad\prob_i(X_n^\fa\in 2C)=1.$$ In other words, the Markov chain with deterministic initial condition $X_0^\fa=i$ reaches $2C$ within $[\log_2(\rho(i))]+1$ steps and remains in this set thereafter.
\end{Le}
\begin{proof}
Recall that since $\supp(\fa)\subseteq C$, for any $j\in\Z^s$ we obtain

$$a_{i-2^nj}^{(n)}\neq0\quad\Longrightarrow\quad \rho(i-2^nj)\le 2^n.$$
Thus the fact that $\rho(i)/2^n\le1$ renders the right hand side of

$$\prob_i(X^\fa_n\in\Z^s\setminus 2C)=\sum_{\rho(j)>2}a_{i-2^nj}^{(n)}$$
an empty sum, since $\rho(i-2^nj)\le2^n$ implies $$\rho(j)\le\rho(i-2^nj)/2^n+\rho(i)/2^n\le2.$$
\end{proof}

\begin{Th}\label{Th:Lp}
Let $p\in[1,\infty)$. Suppose the characteristic Markov chain $X_n^\fa$ of $\mask$ with deterministic initial condition $\ell\in\Z^s$ possesses a stationary distribution $\pi$ in the sense that for all $j\in\Z^s$, $|a^{(n)}_{\ell-2^nj}-\pi_j|\to0$ as $n\to\infty$. Then $X_n^\fa$ converges in $L^p(\Omega,\prob_\ell;\R^s)$ if and only if there is $k\in\Z^s$ such that $\pi=\delta_{k}$. In this case,
$$E_\ell(\|X_n^\fa-k\|^p)\to0\quad\text{as }n\to\infty.$$
\end{Th}
\begin{proof}
Let $\rho$ denote the Minkowski functional of a balanced, closed and convex set containing $\supp(\fa)$. Moreover, for $n\in\N$ define $$A_n=\{(i,j)\in\Z^s\times\Z^s\mid\max(\rho(\ell-2^nj),\rho(j-2^ni))\le2^n-1\}.$$
Note that $(i,j)\in A_n$ implies that $\rho(j)\le1+\frac{\rho(\ell)-1}{2^n}$ as well as $\rho(i)\le1+\frac{\rho(\ell)-1}{2^{2n}}$. Thus there is a bounded set $B$ such that
$$\bigcup_{n\in\N}A_n\subseteq B.$$
Moreover we have
\begin{equation}\begin{split}\int_\Omega\|X_{2n}^\fa(\omega)-X_n^\fa(\omega)\|^p\prob_\ell(d\omega)&=\sum_{i,j\in\Z^s}\|i-j\|^p\prob_\ell(X_{2n}^\fa=i\wedge X_n^\fa=j)\\
&=\sum_{i,j\in\Z^s}\|i-j\|^p\prob_\ell(X_{2n}^\fa=i|X_n^\fa=j)\prob_\ell(X_n^\fa=j)\\
&=\sum_{(i,j)\in A_n}\|i-j\|^pa^{(n)}_{j-2^ni}a^{(n)}_{\ell-2^nj}.\end{split}\end{equation}
Certainly, since $B$ is bounded, the sequence
$$\eps_n:=\sup_{(i,j)\in B}(|\pi_i-a_{j-2^ni}^{(n)}|+|\pi_j-a^{(n)}_{\ell-2^nj}|)$$
converges to zero as $n\to\infty$. Consequently, we obtain $E_\ell(\|X_{2n}^\fa-X_n^\fa\|^p)\geq c_n$, where

\begin{equation}\label{eqn:noncon}\begin{split}c_n&=\sum_{(i,j)\in A_n}\|i-j\|^p\pi_i\pi_j\\
&\quad-\eps_n\sum_{(i,j)\in A_n}\|i-j\|^p(a^{(n)}_{j-2^{n}i}+a^{(n)}_{\ell-2^nj})\\
&\quad-\eps_n^2\sum_{(i,j)\in A_n}\|i-j\|^p.
\end{split}\end{equation}
Thus, whenever there are integers $i\neq j$ such that $\pi_i>0$ and $\pi_j>0$, Equation \eqref{eqn:noncon} implies that
$E_\ell(\|X_{2n}^\fa-X_n^\fa\|)$ is bounded away from zero asymptotically. Hence for $L^p$-convergence of $X_n^\fa$ we need the existence of some $k\in\Z^s$ with $\pi_i=\delta_{ki}$.

Conversely, assume that $\pi=\delta_{\{k\}}$. Then since $X_n^\fa\to k$ in distribution, $X_n^\fa\to k$ in probability. From Lemma \ref{Le:Ball} we conclude that there is $M>0$ such that $\|X_n^\fa\|\le M$ holds $\prob_\ell$-almost surely for $n\in\N$. Hence for $\delta>0$,

\begin{align*}
E_\ell(\|X_n^\fa-k\|^p)&\le\int_{\{\|X_n^\fa-k\|\geq\delta\}}\|X_n^\fa-k\|^pd\prob_\ell+\int_{\{\|X_n^\fa-k\|<\delta\}}\|X_n^\fa-k\|^pd\prob_\ell\\
&\le(M+k)^p\prob_\ell(\|X_n^\fa-k\|\geq\delta)+\delta^p,
\end{align*} 
showing that $E_\ell(\|X_n^\fa-k\|)$ converges to zero.
\end{proof}

Suppose now that the subdivision scheme associated to $\mask$ converges. Then a refinable function $\varphi$ satisfying \eqref{eqn:refine} and \eqref{eqn:pou} exists. Substituting $i\in\Z^s$ for $t$ in 

$$\varphi(t)=\sum_ja_j\varphi(2t-j)$$
and exploiting the fact that $\sum_i\varphi(i)=1$, we observe that $\pi_i=\varphi(-i)$ is a stationary distribution for $X_n^\fa$. Moreover recall that a convergent scheme is called \emph{interpolatory} if and only if there is $k\in\Z^s$ such that for $j\in\Z^s$, $\varphi(j)=\delta_{kj}$. Now Theorem \ref{Th:Lp} translates to the language of refinement schemes as follows:

\begin{Cor}
Suppose the linear scheme associated to $\mask$ converges, and $p\in[1,\infty)$. Then the characteristic Markov chain of $X^\fa_n$ with deterministic initial condition $\ell\in\Z^s$ converges in $L^p(\Omega,\prob_\ell;\R^s)$ if and only if the scheme is interpolatory. In this case the $L^p$-limit is a constant lattice point.
\end{Cor}

\bibliographystyle{amsplain}

\bibliography{Markov}

\providecommand{\bysame}{\leavevmode\hbox to3em{\hrulefill}\thinspace}
\providecommand{\MR}{\relax\ifhmode\unskip\space\fi MR }
\providecommand{\MRhref}[2]{%
  \href{http://www.ams.org/mathscinet-getitem?mr=#1}{#2}
}
\providecommand{\href}[2]{#2}
\begin{thebibliography}{10}

\bibitem{Alex}
A.D. Alexandrov, \emph{{ A theorem on triangles in a metric space and some of
  its applications}}, Trudy Matematicheskogo Instituta imeni V. A. Steklova
  \textbf{38} (1951), 5--23.

\bibitem{Alex1}
\bysame, \emph{{ \"{U}ber eine Verallgemeinerung der Riemannschen Geometrie}},
  { Schriftenreihe des Forschungsinstituts f\"{u}r Mathematik bei der Deutschen
  Akademie der Wissenschaften zu Berlin} \textbf{1} (1957), 33--84.

\bibitem{Ballmann}
W.~Ballmann, \emph{Lectures on spaces of nonpositive curvature},
  Birkh\"{a}user, 1995.

\bibitem{CDM}
A.~S. Cavaretta, W.~Dahmen, and C.~A. Micchelli, \emph{Stationary subdivision},
  American Mathematical Society, 1991.

\bibitem{Ebner}
O.~Ebner, \emph{Convergence of refinement schemes on metric spaces},
  Proceedings of the American Mathematical Society (to appear).

\bibitem{Grohs}
P.~Grohs, \emph{A general proximity analysis of nonlinear subdivision schemes},
  SIAM Journal on Mathematical Analysis \textbf{42} (2010), no.~2, 729--750.

\bibitem{Lang}
S.~Lang, \emph{Fundamentals of differential geometry}, Graduate Texts in
  Mathematics, vol. 191, Springer, 1999.

\bibitem{Prautzsch}
C.~A. Micchelli and H.~Prautzsch, \emph{{ Uniform refinement of curves}},
  Linear Algebra and Applications \textbf{114/115} (1989), 841--870.

\bibitem{Reif}
U.~Reif and J.~Peters, \emph{Subdivision surfaces}, Geometry and Computing,
  vol.~3, Springer, 2008.

\bibitem{Sturm2}
K.-T. Sturm, \emph{{Nonlinear Martingale Theory for Processes with Values in
  Metric Spaces of Nonpositive Curvature}}, The Annals of Probability
  \textbf{30} (2002), no.~3, 1195--1222.

\bibitem{Sturm}
\bysame, \emph{{Probability measures on metric spaces of nonpositive
  curvature}}, Heat kernels and analysis on manifolds, graphs, and metric
  spaces, Contemporary Mathematics, vol. 338, American Mathematical Society,
  2003, pp.~357--390.

\bibitem{WNW}
J.~Wallner, E.~Nava Yazdani, and A.~Weinmann, \emph{{Convergence and smoothness
  analysis of subdivision rules in Riemannian and symmetric spaces}}, Advances
  in Computational Mathematics \textbf{34} (2011), no.~2, 201--218.

\bibitem{Zhou2}
X.~Zhou, \emph{{ Subdivision schemes with nonnegative masks}}, Mathematics of
  Computation \textbf{74} (2005), 819--839.

\bibitem{Zhou}
\bysame, \emph{{ On multivariate subdivision schemes with nonnegative finite
  masks}}, Proceedings of the American Mathematical Society \textbf{134}
  (2006), 859--869.

\end{thebibliography}

\end{document}